\documentclass[a4paper,10pt]{article}
\usepackage{microtype}
\usepackage[dvips]{graphicx}
\usepackage[english,italian]{babel}
\usepackage{amsfonts}
\usepackage{amssymb}
\usepackage{amsmath}
\usepackage{amsthm}
\usepackage{MnSymbol}
\usepackage{enumerate}
\usepackage{leftidx}
\usepackage{geometry}
\usepackage{psfrag}
\usepackage{float}
\newcommand{\G}{\Gamma}
\renewcommand\mod{\operatorname{mod}}
\newcommand{\ZZ}{\mathbb{Z}}
\newcommand{\cB}{\mathcal B}

\newcommand{\cL}{\mathcal L}

\title{New results on path-decompositions\\ and their down-links}

\author{
\begin{otherlanguage*}{italian}
A. Benini, L. Giuzzi and A. Pasotti\thanks{
anna.benini@ing.unibs.it, luca.giuzzi@ing.unibs.it,
anita.pasotti@ing.unibs.it,
Dipartimento di Ma\-te\-ma\-ti\-ca,
Facolt\`a di Ingegneria,
Universit\`a degli Studi di Brescia,
Via Valotti 9,
I-25133 Brescia (IT).
\vskip.2pt
The present research was performed within the activity of GNSAGA
of the Italian INDAM with the financial support of the Italian
Ministry MIUR, projects ``Strutture di incidenza e combinatorie''
 and ``Disegni combinatorici, grafi e loro applicazioni''.
}\end{otherlanguage*}}

\date{}
\newtheorem{defi}{Definition}[section]
\newtheorem{prop}[defi]{Proposition}
\newtheorem{lem}[defi]{Lemma}
\newtheorem{rem}[defi]{Remark}

\newtheorem{thm}[defi]{Theorem}
\newtheorem{cor}[defi]{Corollary}

\begin{document}
\selectlanguage{english}
\maketitle
\selectlanguage{english}
\begin{abstract}
In \cite{BGP} the concept of \emph{down-link} from a
$(K_v,\Gamma)$-design $\cB$ to a $(K_n,\Gamma')$-design $\cB'$ has been
introduced. In the present paper  the spectrum problems for
$\Gamma'=P_4$ are studied. General results on the existence of path-decompositions
and embeddings between path-decompositions playing a fundamental role for the construction of down-links
are also presented.
%In a previous paper the new concept of \emph{down-link} from a
%$(K_v,\Gamma)$-design $\cB$ to a $(K_n,\Gamma')$-design $\cB'$ was
%defined as a map $f:\cB\to \cB'$ mapping any block of $\cB$ into
%one of its subgraphs and the problem to guarantee the  existence of a down-link was studied ......
%In the present paper  the spectrum problem in the case
%$\Gamma'=P_4$ is studied. We show that it is always possible to embed any $(K_n,P_4)$-design into a $(K_m,P_4)$-design
%for any admissible $m>n+1$ and to find a down-link to a $P_4$-design of order at most $v+3$....
\end{abstract}


\section{Introduction}
Suppose $\G\leq K$ to be a subgraph of $K$. A $(K,\Gamma)$-\emph{design}, or \emph{$\Gamma$-decomposition} of $K$,
is a set of graphs isomorphic to $\G$ whose edges partition the edge set of $K$.
%Let $K$ be a graph and $\Gamma\leq K$.
%A $(K,\Gamma)$-\emph{design}, also called a \emph{$\Gamma$-decomposition} of $K$,
%is a set $\cB$ of graphs all isomorphic to $\Gamma$,
%called \emph{blocks}, partitioning the edge-set
%of $K$.
Given a graph $\Gamma$, the problem of determining the existence
of  $(K_v,\Gamma)$-designs, also called \emph{$\Gamma$-designs of
order $v$}, where $K_v$ is the complete graph on $v$
vertices,
has been extensively studied;
see the surveys \cite{BO,BZ}.
In \cite{BGP} we proposed the following definition.
\begin{defi}
\label{d1}
  Given a $(K,\Gamma)$-design $\cB$ and a $(K',\Gamma')$-design $\cB'$
  with $\Gamma'\leq \Gamma$, a \emph{down-link} from $\cB$ to $\cB'$
  is a function $f: \cB \rightarrow \cB'$
  such that $f(B)\leq B$, for any $B \in \cB$.
\end{defi}
When such a  function $f$  exists, we say that it is possible to
\emph{down-link} $\cB$ to $\cB'$.
\par
%In this paper we shall investigate down-links between a $(K_v,\G)$-design
%and a $P_4$-design. In order to get results on such a kind of down-link,
%firstly we study path-designs and their embeddings. In particular,
%in Section \ref{spectrum} we recall the spectrum problems on the existence of
%down-links introduced in \cite{BGP}. In Section \ref{existence}
%we present a result on $P_k$-decompositions of the complete
%bipartite graph and a sufficient condition for the existence of a $P_4$-decomposition
%of an arbitrary graph. These results allow us to get embeddings and down-links
%in Section \ref{embedding}. More in detail,.....

As seen in \cite{BGP}, down-links are closely related to metamorphoses
\cite{LR1}, their
generalizations \cite{LMQ} and embeddings \cite{Q2002a}.
In close analogy to embeddings, we
introduced spectrum problems about down-links:

\begin{enumerate}[(I)]
\item  For each admissible $v$, determine the set
  ${\cL}_1\G(v)$ of all integers $n$ such that
  there exists \emph{some} $\G$-design of order $v$ down-linked
  to a $\G'$-design of order $n$.
\item For each admissible $v$, determine the set ${\cL}_2\G(v)$ of all
  integers $n$ such that
  \emph{every} $\G$-design of order $v$ can be down-linked to a $\G'$-design
  of order $n$.
\end{enumerate}

\noindent
In \cite[Proposition 3.2]{BGP},
we proved that
  for any $v$ such that there exists a
  $(K_v,\G)$-design  and any $\G'\leq\G$,
  the sets $\cL_1\G(v)$ and $\cL_2\G(v)$ are always non-empty.
In the same paper the case $\G'=P_3$ has been investigated in detail.\\
Here
we shall deal with the case $\G'=P_4$. In order to get results about down-links to $P_4$-designs,
we shall first study
path-designs and their embeddings. More precisely, in Section \ref{existence} we determine
sufficient conditions for the existence of $P_4$-decompositions of any graph $\G$
and $P_k$-decompositions of complete bipartite graphs. In Section \ref{embedding}, applying the
results of Section \ref{existence}, we are able to
prove the existence of embeddings and down-links between path-designs.
Section \ref{direct} is devoted to the cases of cycle systems and path-designs,
with general theorems and directed constructions.
\par
Throughout this paper the following standard notations will be used;
see also \cite{Ha}.
For any graph $\Gamma$, write $V(\Gamma)$ for the set of its vertices and
$E(\Gamma)$ for the set of its edges. If $\cB$ is a collection
of graphs, by $V(\cB)$ we will mean the set of the vertices of all its elements.
By $t\Gamma$ we shall denote the disjoint union of $t$ copies of graphs
all isomorphic to $\Gamma$.
As usual, $P_k=[a_1,\ldots,a_k]$ is the path with $k-1$ edges and
$C_k=(a_1,\ldots,a_k)$, $k\geq3$, is the cycle of length $k$.
Also, $K_{m,n}$ is the complete bipartite graph with parts of size $m$ and $n$.
When we focus on the actual parts
$X$ and $Y$, $K_{X,Y}$ will be written.

\section{Existence of some path-designs}\label{existence}
\label{sec:gen}
In this section we present new results on the existence of path decompositions.
Recall that a $(K_n,P_k)$-design
exists if, and only if, $n(n-1)\equiv 0\pmod{2(k-1)}$; see \cite{T}.

\begin{prop}\label{bip}
Let $k$ be an even integer. For $x=k-2,k$
the complete bipartite graph $K_{k-1,x}$ admits a $P_k$-decomposition.
\end{prop}
\begin{proof}
Consider the bipartite graph $K_{A,I}$ where  $A=\{a_1,\dots,a_{k-1}\}$ and $I=\{1,\dots,x\}$ with $x=k-2,k$.\\
Let $U^t=(1,\dots,1)$ be an $\frac{x}{2}$-tuple. Set $P_1^t=(1,\ldots, \frac{x}{2})$ and for $i=1,\ldots, \frac{x}{2}$,
$P_i^t=(i,i+1,\dots,\frac{x}{2},1,2,\dots,i-1)$, $\overline{P}_i^t=P_i^t+\frac{x}{2}U$, $A_i=a_iU$,
$\overline{A}_i=a_{i+\frac{k}{2}}U$.\\
If $k\equiv 0  \pmod 4$, consider the $\frac{x}{2}\times k$ matrices
\begin{align*}
M & =(P_1~A_1~\overline P_1~A_2~P_2~A_3~\overline
P_2\dots P_i~A_{2i-1}~\overline P_i~A_{2i}\dots P_{\frac{k}{4}}~A_{\frac{k-2}{2}}~\overline P_{\frac{k}{4}}~A_{\frac{k}{2}})\\
\overline M & =(\overline P_1~\overline A_1~P_1~\overline A_2~\overline P_2~\overline A_3~P_2\dots
\overline P_i~\overline A_{2i-1}~P_i~\overline A_{2i}\dots \overline P_{\frac{k}{4}}~\overline A_{\frac{k-2}{2}}~P_{\frac{k}{4}}~A_{\frac{k}{2}}).
\end{align*}
If $k\equiv 2  \pmod 4$, consider the $\frac{x}{2}\times k$ matrices
\begin{align*}
M & =(P_1~A_1~\overline P_1~A_2~P_2~A_3~\overline
P_2\dots P_i~A_{2i-1}~\overline P_i~A_{2i}\dots P_{\frac{k+2}{4}}~A_{\frac{k}{2}})\\
\overline M & =(\overline P_1~\overline A_1~P_1~\overline A_2~\overline P_2~\overline A_3~P_2\dots
\overline P_i~\overline A_{2i-1}~P_i~\overline A_{2i}\dots \overline P_{\frac{k+2}{4}}~A_{\frac{k}{2}}).
\end{align*}
In either case, the rows of $M$ and $\overline M$, taken together, are the $x$ paths of a $P_k$-decomposition of $K_{A,I}$.
\end{proof}

\begin{thm}\label{every}
Let $\Gamma$ be a graph with at least two vertices of degree $|V(\G)|-1$. Then  $\Gamma$ admits a  $P_4$-decomposition
 if, and only if, $|E(\Gamma)|\equiv 0\pmod 3$. If $|E(\Gamma)|\equiv 1,2\pmod 3$, then $\Gamma$ can be partitioned into
 a $P_4$-decomposition together with one or two (possibly connected) edges, respectively.
\end{thm}
\begin{proof}
The condition is obviously necessary. For sufficiency, let
%$A=\{\alpha,\beta\}$, where
 $\alpha$ and $\beta$ be two vertices of degree $|V(\G)|-1$. Delete $\alpha$ and $\beta$ in $\Gamma$,
 as to obtain a graph $G$.
% Denote by $\bar{G}$ the graph formed by the removed
%edges, that is the graph $K_{A,H}\cup [\alpha,\beta]$ where we set $H=V(\Gamma)\setminus A$.
Let $G'$ be a maximal $P_4$-decomposable subgraph of $G$ and remove from $G$ the edges of $G'$, determining a new graph $G''$.
%$\Gamma$ to obtain the graph $\Gamma'$. \\
%$\bullet$~The case $G'=G$.\\
%It is enough  to show that the $\bar{G}$ admits a $P_4$-decomposition. Firstly, we observe that
%$|E(\bar{G})|\equiv 0\pmod 3$ implies $2(v-2)+1\equiv 0\pmod 3$, so $v-2\equiv 0\pmod 3$. From Lemma \ref{l1}
%we have a $P_4$-decomposition
%$\bullet$~The case $G'\neq G$.\\
In general, $G''$  is not connected and its
connected components are either isolated  vertices or  stars  or cycles of length $3$;
call $\cal I$, $\cal S$ and $\cal C$ their (possibly empty) sets.
Let $\Gamma'$ be the graph obtained removing the edges of $G'$ from $\Gamma$.
Clearly, $|E(\Gamma)|\equiv 0\pmod 3$ implies  $|E(\Gamma')|\equiv~0\pmod3$; thus it remains to show that $E(\Gamma')$ is $P_4$-decomposable.
Obviously $\alpha$ and $\beta$ are of degree $|V(\G)|-1$ also in $\Gamma'$.
Let $A=\{\alpha,\beta\}$ and consider the following decomposition  $\Gamma'=K_A\cup K_{A,{\cal I}}\cup({\cal C}\cup K_{A,V({\cal C})})\cup ({\cal S}\cup K_{A,V({\cal S})})$.
 We begin by providing, separately, $P_4$-decompositions of $K_{A,{\cal I}}$, ${\cal C}\cup K_{A,V({\cal C})}$ and ${\cal S}\cup K_{A,V({\cal S})}$.

$i$)~It is easy to see that for any $3$-subset of $\cal I$, say $H_3$,
the graph $K_{A,H_3}$ has a $P_4$-decomposition. Thus, depending on the congruence class modulo $3$ of $|{\cal I}|$, $K_{A,{\cal I}}$ can be partitioned into a
$P_4$-decomposition together with the following possible remnants.
\begin{table}[h]
\begin{small}
\begin{center}
\begin{tabular}{c|c|c}
\hline
      $(i_1)$ \ {\scriptsize $|{\cal I}|\equiv 0\pmod 3$} &  $(i_2)$ \ {\scriptsize $|{\cal I}|\equiv 1 \pmod 3$}&     $(i_3)$ \ {\scriptsize
      $|{\cal I}|\equiv 2 \pmod 3$} \cr
      \hline
       the set $\emptyset$  & the path $[\alpha,h,\beta]$ & the cycle $(h_1,\alpha,h_2,\beta)$  \cr
  &with $h\in {\cal I}$ &  with $h_1,h_2\in {\cal I}$ \cr
 \hline
\end{tabular}
\end{center}
\end{small}
\caption{Case $i$.}
\end{table}

\smallskip

$ii$)~For any $3$-cycle $C \in\cal C$, the graph $C\cup K_{A,V(C)}$ has a $P_4$-decomposition. Thus,
${\cal C}\cup K_{A,V({\cal C})}$ also admits a $P_4$-decomposition.

\smallskip

$iii$)~It is not difficult to see that, for any star $S_c\in {\cal S}$ of center $c$, the graph
$S_c\cup K_{A,V(S_c)}$ has a partition into a $P_4$-decomposition together with either the path $[\alpha,c,\beta]$ or the graph
$(\alpha,c,\beta,v)\cup [c,v]$, where $v$ is any
external vertex, depending on whether the number of vertices of $S_c$ is odd  or even. Let ${\cal S}_1$
(respectively ${\cal S}_2$)
be the set of stars with an odd (even) number of vertices.
For any three stars of ${\cal S}_1$ (${\cal S}_2$) the remnants give $P_4$-decomposable graphs. So
${\cal S}_1\cup K_{A,V({\cal S}_1)}$, as well as ${\cal S}_2\cup K_{A,V({\cal S}_2)}$,  can be partitioned into a $P_4$-decomposition together with
the possible remnants outlined in Tables 2 and 3.
\begin{table}[H]
\begin{small}
\begin{center}
\begin{tabular}{c|c|c}
\hline
     % $(iii_{11})$ \  {\scriptsize $|{\cal S}_1|\equiv 0\pmod 3$} &  $(iii_{12})$ \ {\scriptsize $|{\cal S}_1|\equiv 1\pmod 3$}&
     %  $(iii_{13})$ \ {\scriptsize $|{\cal S}_1|\equiv 2\pmod 3$} \cr
      $(iii_{11})$  &  $(iii_{12})$ & $(iii_{13})$ \cr
 {\scriptsize $|{\cal S}_1|\equiv 0\pmod 3$} &  {\scriptsize $|{\cal S}_1|\equiv 1\pmod 3$} & {\scriptsize $|{\cal S}_1|\equiv 2\pmod 3$}\cr
      \hline
  & the path $[\alpha,c,\beta]$ & the cycle $(c_1,\alpha,c_2,\beta)$\cr
$\emptyset$  & where & where \cr
  &$c$ is the center of a star  &  $c_1,c_2$ are centers of two stars \cr
 \hline
\end{tabular}
\end{center}
\end{small}
\caption{Case $iii_1$: ${\cal S}_1\cup K_{A,V({\cal S}_1)}$. }
\end{table}
\begin{table}[h]
\begin{small}
\begin{center}
\begin{tabular}{c|c|c}
\hline
$(iii_{21})$  &  $(iii_{22})$ &
 $(iii_{23})$ \cr
 {\scriptsize $|{\cal S}_2|\equiv 0\pmod 3$} &  {\scriptsize $|{\cal S}_2|\equiv 1\pmod 3$} & {\scriptsize $|{\cal S}_2|\equiv 2\pmod 3$}\cr
\hline
&the graph & the graph \cr
% &{\scriptsize$(\alpha,c,\beta,v)\cup [c,v]$}&{\scriptsize$(\alpha,c_1,\beta,v_1)\cup [c_1,v_1]\cup(\alpha,c_2,\beta,v_2)\cup [c_2,v_2]$} \cr
 & {\scriptsize$(\alpha,c,\beta,v)\cup [c,v]$} & {\scriptsize$\bigcup_{i=1}^{2}(\alpha,c_i,\beta,v_i)\cup [c_i,v_i]$}\cr
 %$K_{A,{\{c_1,c_2,v_1,v_2\}}}\cup[c_1,v_1]\cup[c_2,v_2]$
 $\emptyset$& where $c$ is the center & where $c_1,c_2$ are centers \cr
 & and $v$ is an external vertex & and $v_1,v_2$ are external vertices \cr
& of a star  &  of two stars \cr
 \hline
\end{tabular}
\end{center}
\end{small}
\caption{Case $iii_2$: ${\cal S}_2\cup K_{A,V({\cal S}_2)}$.}
\end{table}
\noindent
The remnants from  $i)$, $iii_1)$ and $iii_2)$ together with the edge $[\alpha,\beta]$ can be combined in
$27$ different ways to obtain $27$ connected graphs with  $t$ edges.
It is a routine to check that we have exactly $9$ cases with $t\equiv i\pmod 3$,
for $i=0,1,2$.

\noindent In Table 4 we will list in detail the $9$ cases with  $t\equiv 0\pmod 3$ and,
for each of them, in Table 5 we give the corresponding graph.
\begin{table}[H]
{\small \begin{tabular}{c}
 \begin{tabular}{|l|l|l|l|}
\hline
&$i$&$iii_1$&$iii_2$ \cr
\hline
$a_1$&$\emptyset$& $\emptyset$ &$iii_{22} $\cr
\hline
$a_2$&$\emptyset$& $iii_{13}$ &$iii_{23} $ \cr
\hline
$a_3$&$\emptyset$& $iii_{12}$ &$\emptyset$ \cr
\hline
\end{tabular}~
\begin{tabular}{|l|l|l|l|}
\hline
&$i$&$iii_1$&$iii_2$ \cr
\hline
$a_4$&$i_2$& $\emptyset$ &$\emptyset$\cr
\hline
$a_5$&$i_2$& $iii_{13}$ &$iii_{22} $ \cr
\hline
$a_6$&$i_2$& $iii_{12}$ &$iii_{23} $ \cr
\hline
\end{tabular}~
\begin{tabular}{|l|l|l|l|}
\hline
&$i$&$iii_1$&$iii_2$ \cr
\hline
$a_7$&$i_3$& $\emptyset$ &$iii_{23} $\cr
\hline
$a_8$&$i_3$& $iii_{13}$ &$\emptyset$ \cr
\hline
$a_9$&$i_3$& $iii_{12}$ &$iii_{22} $ \cr
\hline
\end{tabular}
\end{tabular}}
\caption{$t\equiv 0 \pmod 3$.}
\end{table}
%From previous  cases we obtain  the following $5$  graphs (remember we have to add the edge $[\alpha,\beta]$).

\begin{figure}[htbp]
\psfrag{a}{$\alpha$}
\psfrag{b}{$\beta$}
\centering
\begin{minipage}[c]{.25\textwidth}
\begin{center}
%\centering\setlength{\captionmargin}{0pt}%
\includegraphics[width=.35\textwidth]{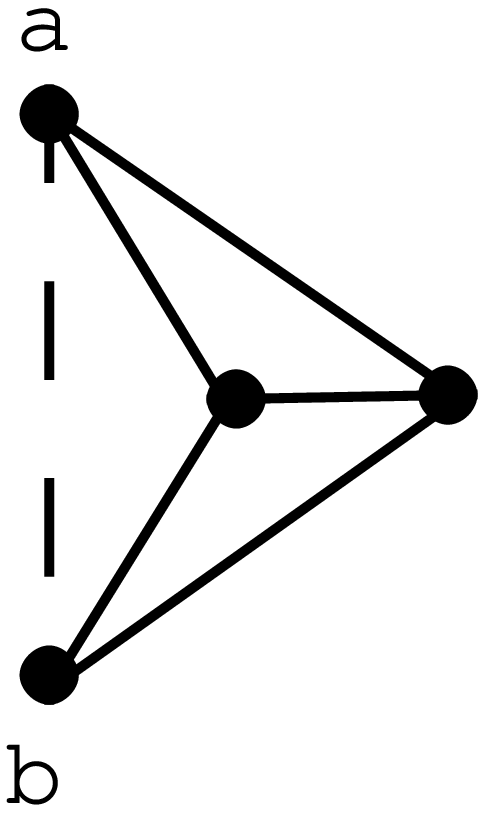}
\caption{Case $a_1$.}\label{Fig1}
\end{center}
\end{minipage}%
\hspace{5mm}%
\begin{minipage}[c]{.25\textwidth}
\begin{center}
%\centering\setlength{\captionmargin}{0pt}%
\includegraphics[width=.65\textwidth]{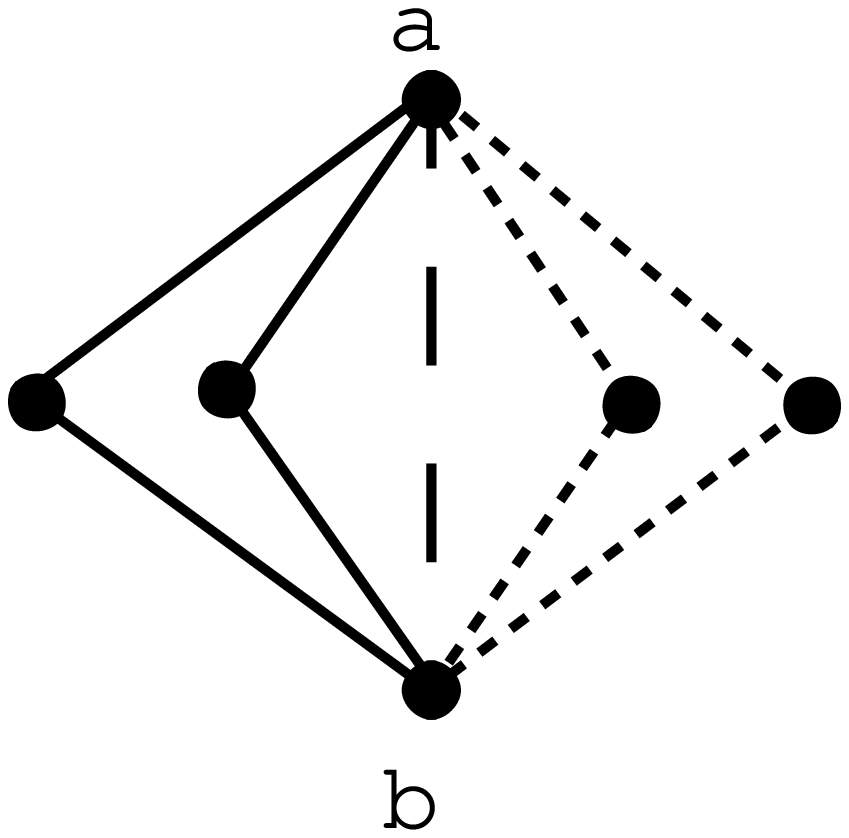}
\caption{Case $a_8$.}\label{Fig2}
\end{center}
\end{minipage}%
\hspace{5mm}%
\begin{minipage}[c]{.40\textwidth}
\begin{center}
%\centering\setlength{\captionmargin}{0pt}%
\includegraphics[width=.5\textwidth]{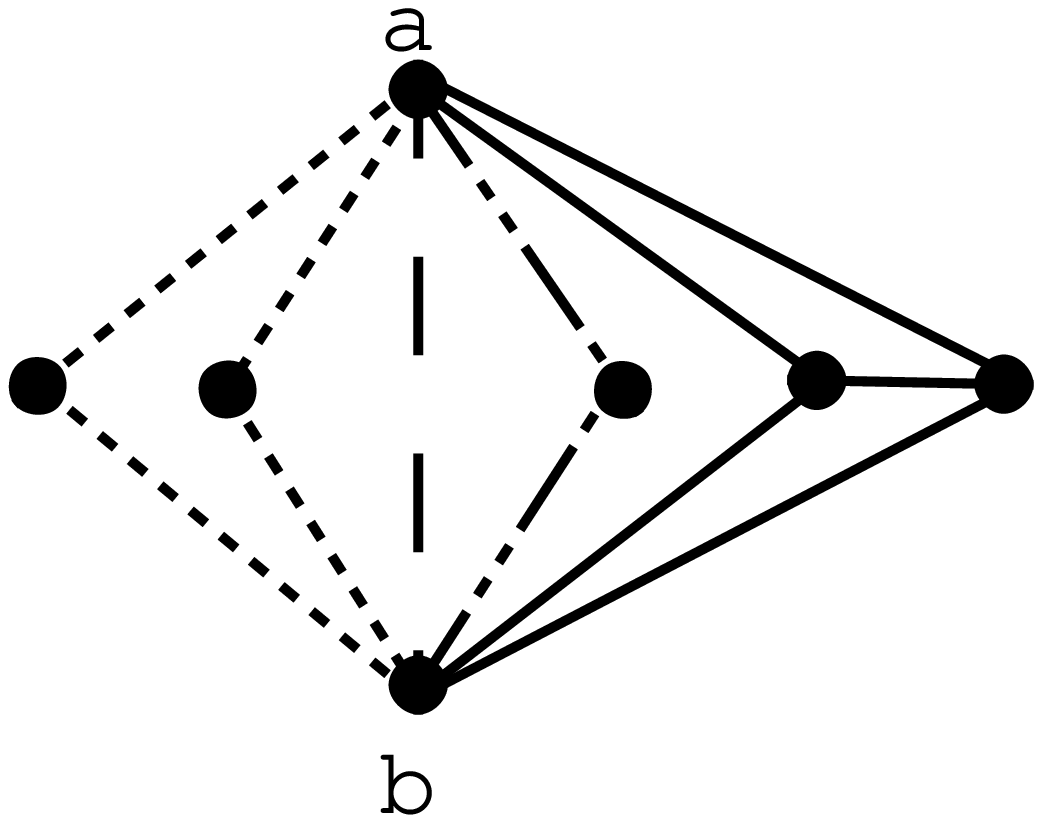}
\caption{Cases $a_5$ and $a_9$.}\label{Fig3}
\end{center}
\end{minipage}%
\end{figure}
\begin{figure}[h]
\psfrag{a}{$\alpha$}
\psfrag{b}{$\beta$}
\psfrag{c}{$c$}
\psfrag{v}{$v$}
\psfrag{h}{$h$}
\centering
\begin{minipage}[c]{.40\textwidth}
\begin{center}
%\centering\setlength{\captionmargin}{0pt}%
\includegraphics[width=.85\textwidth]{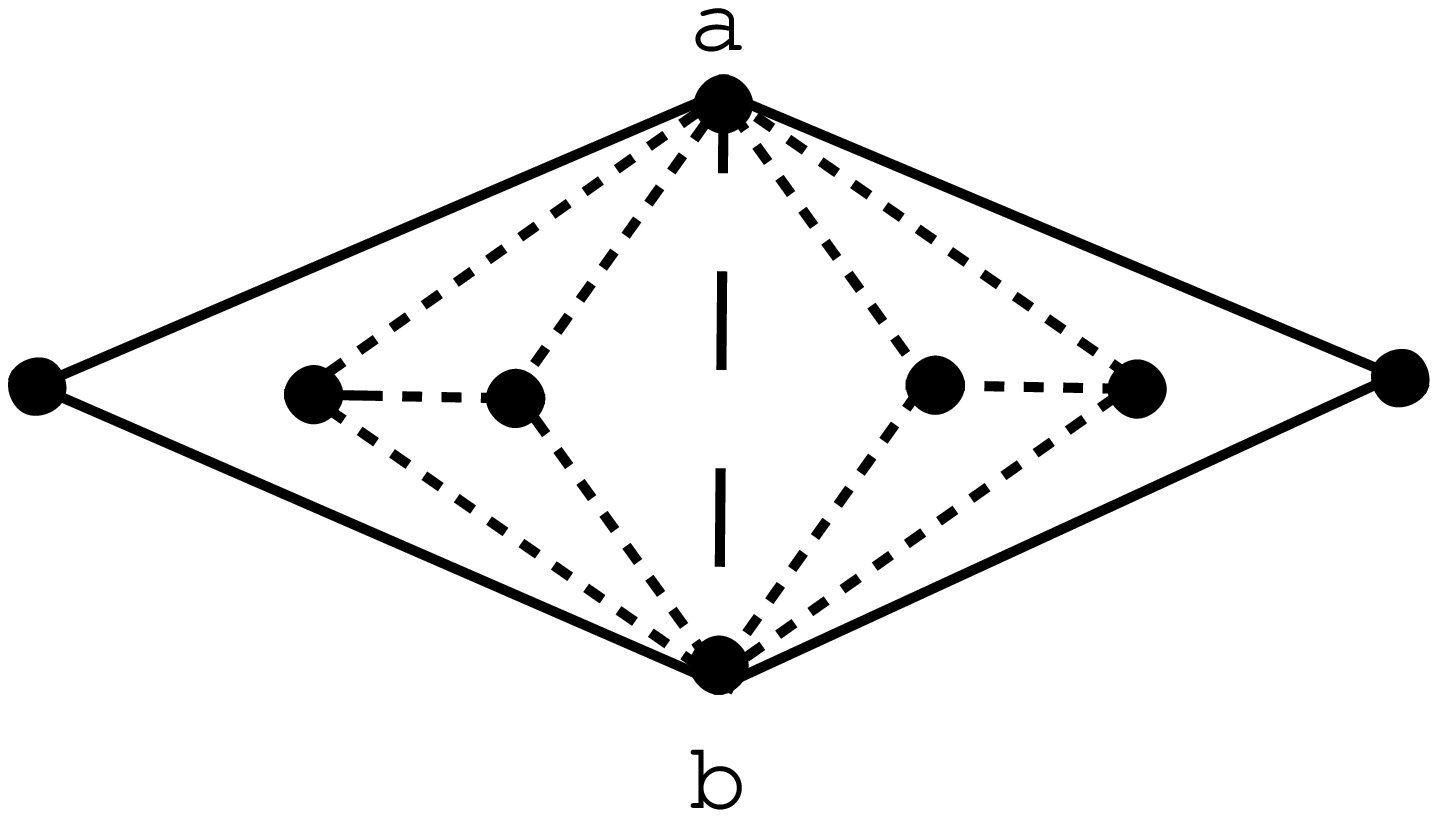}
\caption{Cases $a_2$, $a_6$ and $a_7$.}\label{Fig4}
\end{center}
\end{minipage}
\hspace{5mm}
\begin{minipage}[c]{.40\textwidth}
\begin{center}
%\centering\setlength{\captionmargin}{0pt}%
\includegraphics[width=0.22\textwidth]{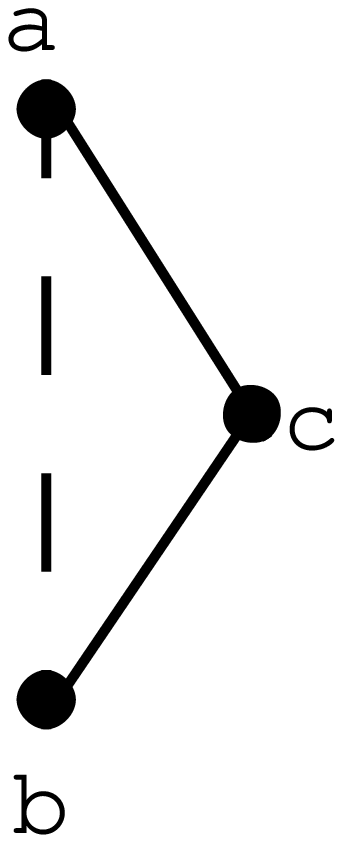}\quad\quad\quad\quad
\includegraphics[width=0.22\textwidth]{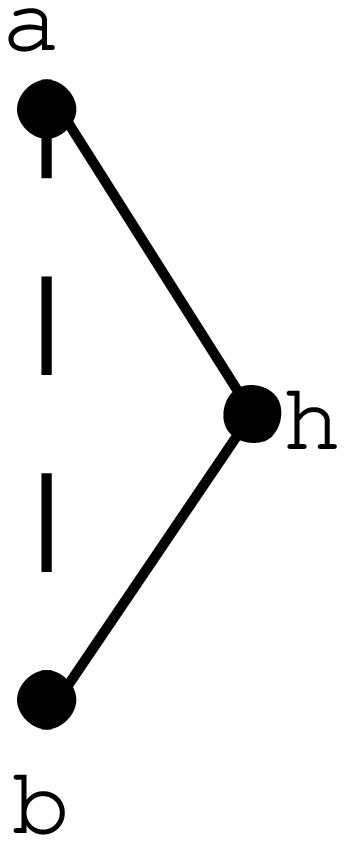}
\caption{Cases $a_3$ and $a_4$.}\label{Fig5}
\end{center}
\end{minipage}
\begin{center}
Table 5: Graphs of the remnants plus edge $[\alpha,\beta]$.
\end{center}
\end{figure}
\noindent
It is easy to determine a $P_4$-decomposition of the graphs in Figures \ref{Fig1}, \ref{Fig2}, \ref{Fig3}, \ref{Fig4}.
In cases $a_3$ and $a_4$ (Figure \ref{Fig5}) a $P_4$-decomposition is clearly not possible,
thus we proceed back tracking one step in the construction. How to deal with case $a_3$
is explained in Figure \ref{Resti1}.
\begin{figure}[h]
\psfrag{a}{$\alpha$}
\psfrag{b}{$\beta$}
\psfrag{c}{$c$}
\psfrag{v}{$v$}
\psfrag{h}{$h$}
\psfrag{v1}{$v_1$}
\psfrag{v2}{$v_2$}
\psfrag{v3}{$v_3$}
\psfrag{a1}{$a_1$}
\psfrag{a2}{$a_2$}
\psfrag{a3}{$a_3$}
\psfrag{h1}{$h_1$}
\psfrag{h2}{$h_2$}
\psfrag{h3}{$h_3$}
%%\emph{Testo sopra l'immagine.}\\
\begin{minipage}{.30\textwidth}
\centering
\includegraphics[width=.80\textwidth]{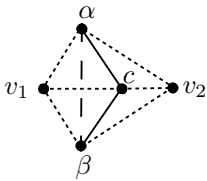}\\
%\caption{}
%\emph{\small Didascalia}\bigskip
\end{minipage}
\begin{minipage}{.70\textwidth}
\caption{We recover the two $P_4$'s from 2 radii of the star of center $c$.}\label{Resti1}
\end{minipage}\\
%\emph{Testo sotto l'immagine.}
\end{figure}

\noindent
In case $a_4$ we have to distinguish several subcases depending on the size of
$\cal{I,C}$ and $\cal S$.
When $|{\cal I}|>1$ see Figure \ref{Resti2}. For $|{\cal I}|=1$ and $|{\cal C}|\neq 0$,
see Figure \ref{Resti4}.
\begin{figure}[H]
\psfrag{a}{$\alpha$}
\psfrag{b}{$\beta$}
\psfrag{c}{$c$}
\psfrag{v}{$v$}
\psfrag{v1}{$v_1$}
\psfrag{v2}{$v_2$}
\psfrag{v3}{$v_3$}
\psfrag{a1}{$a_1$}
\psfrag{a2}{$a_2$}
\psfrag{a3}{$a_3$}
\psfrag{h}{$h$}
\psfrag{h1}{$h_1$}
\psfrag{h2}{$h_2$}
\psfrag{h3}{$h_3$}
%\emph{Testo sopra l'immagine.}\\
\begin{minipage}{.30\textwidth}
\centering
\includegraphics[width=.80\textwidth]{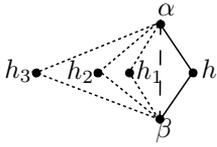}\\
%\caption{}
%\emph{\small Didascalia}\bigskip
\end{minipage}
\begin{minipage}{.70\textwidth}
\caption{If $|{\cal I}|>1$, we recover the two $P_4$'s from 3 vertices of $\cal I$.}\label{Resti2}
\end{minipage}
%\emph{Testo sotto l'immagine.}
\end{figure}
\begin{figure}[H]
\psfrag{a}{$\alpha$}
\psfrag{b}{$\beta$}
\psfrag{c}{$c$}
\psfrag{v}{$v$}
\psfrag{v1}{$v_1$}
\psfrag{v2}{$v_2$}
\psfrag{v3}{$v_3$}
\psfrag{b1}{$b_1$}
\psfrag{b2}{$b_2$}
\psfrag{b3}{$b_3$}
\psfrag{h}{$h$}
\psfrag{h1}{$h_1$}
\psfrag{h2}{$h_2$}
\psfrag{h3}{$h_3$}
%\emph{Testo sopra l'immagine.}\\
\begin{minipage}{.30\textwidth}
\centering
\includegraphics[width=.80\textwidth]{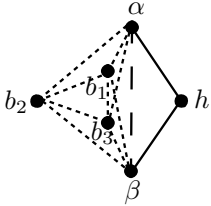}\\
%\caption{}
%\emph{\small Didascalia}\bigskip
\end{minipage}
\begin{minipage}{.70\textwidth}
\caption{If $|{\cal I}|=1$ and $|{\cal C}|\neq 0$ we recover the three $P_4$'s from a $C_3$.}\label{Resti4}
\end{minipage}
\end{figure}
\noindent
When $|{\cal I}|=1$ and $|{\cal C}|= 0$ we have two possibilities. If there is
one star of $\cal S$ with at least two edges, we proceed as explained in Figure  \ref{Fig2radii}.
\begin{figure}[H]
\psfrag{a}{$\alpha$}
\psfrag{b}{$\beta$}
\psfrag{c}{$c$}
\psfrag{v}{$v$}
\psfrag{v1}{$v_1$}
\psfrag{v2}{$v_2$}
\psfrag{v3}{$v_3$}
\psfrag{a1}{$a_1$}
\psfrag{a2}{$a_2$}
\psfrag{a3}{$a_3$}
\psfrag{h}{$h$}
\psfrag{h1}{$h_1$}
\psfrag{h2}{$h_2$}
\psfrag{h3}{$h_3$}
%\emph{Testo sopra l'immagine.}\\
\begin{minipage}{.30\textwidth}
\centering
\includegraphics[width=.80\textwidth]{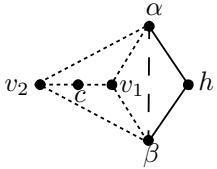}\\
%\caption{}
%\emph{\small Didascalia}\bigskip
\end{minipage}
\begin{minipage}{.70\textwidth}
\caption{If $|{\cal I}|= 1$, $|{\cal C}|= 0$, and $\exists S_c \in{\cal S}$
with $P_3\leq S_c$ we recover the
two $P_4$'s from 2 radii of $S_c$.}\label{Fig2radii}
\end{minipage}
\end{figure}
\noindent
Otherwise, $G''$ consists of an isolated vertex $h$ and a set $\cal P$ of disjoint $P_2$'s.
Since $|E(\G')|\equiv 0 \pmod 3$, the size of $\cal P$ is also divisible by $3$, let $|{\cal P}|=3p$.
It is easy to see that for any $3$-subset of $\cal P$, say $P^3$, the graph $K_{A,P^3}$
has a $P_4$-decomposition. After $p-1$ steps,
the remnant is the graph in Figure \ref{SoloP3}, which likewise admits a $P_4$-decomposition.
This concludes the case $t \equiv 0 \pmod 3$.
\begin{figure}[H]
\psfrag{a}{$\alpha$}
\psfrag{b}{$\beta$}
\psfrag{h}{$h$}
\begin{minipage}{.30\textwidth}
\centering
\includegraphics[width=.80\textwidth]{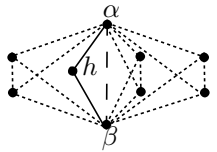}\\
\end{minipage}
\begin{minipage}{.70\textwidth}
\caption{If $|{\cal I}|=1$, $|{\cal C}|=0$ and $S$ is a disjoint union of $P_2$'s
we recover the $15$ edges from the last but one step.}\label{SoloP3}
\end{minipage}
\end{figure}

\noindent
With similar arguments, when $t\equiv 1,2\pmod 3$ it is possible to find a $P_4$-decomposition of $E(\G)$
 leaving as remnants, respectively, one or two edges.
\end{proof}

\section{Embeddings and down-links to $P_4$-designs}\label{embedding}
The results presented in the previous section are used to prove
the existence of embeddings and down-links to path designs.
In particular, we shall focus our attention on
$P_4$-decompositions.

\begin{thm}\label{partial}
Any partial $(K_v,P_4)$-design can be embedded into a $(K_n,P_4)$-design for any admissible $n\geq v+2$.
\end{thm}
\begin{proof}
Let $\cB$ be a partial $(K_v,P_4)$-design.
Let $A$ be a set of vertices disjoint from $V(K_v)$ with $v+|A|\equiv 0,1 \pmod 3$ and $|A|\geq 2$.
Let  $\Gamma$ be the graph such that $V(\Gamma)=V(K_v)\cup A$ and $E(\Gamma)=E(K_{v+|A|})\setminus E({\cal B})$.
 Since $|A|\geq 2$, by Theorem \ref{every} there exists a $(\Gamma,P_4)$-design
$\cB'$ and, clearly, $\cB\cup \cB'$ is a  $(K_{v+|A|,P_4})$-design.
\end{proof}

\begin{cor}\label{n>v+1}
For any  $(K_v,\Gamma)$-design with $P_4\leq \Gamma$
 $$\{n\geq v+2 \  | \ n\equiv0,1\pmod3\}\subseteq
{\cal L}_2\Gamma(v)\subseteq {\cal L}_1\Gamma(v).$$
\end{cor}
\begin{proof}
Let $\cB$ be a $(K_v,\Gamma)$-design with $P_4\leq \Gamma$.
Choose a $P_4$ in each block of $\cB$ and call $\cal P$ the set of such $P_4$'s. Obviously, $\cal P$ is a partial
$P_4$-decomposition of $K_v$. Hence, by Theorem \ref{partial}, $\cal P$ can be embedded into a $(K_n,P_4)$-design $\cB'$
for any admissible $n\geq v+2$.
The construction also guarantees the existence of a down-link from $\cB$ to $\cB'$.
\end{proof}

\begin{thm}\label{thm:emb}
For any even integer $k$,
a $P_k$-design of order $n\equiv 0,1(\mod k-1)$ can be embedded into a $P_{k}$-design of any order $m>n+1$
with $m\equiv 0,1(\mod k-1)$.
\end{thm}
\begin{proof}
Let ${\cB}$ be a $(K_n,P_k)$-design with $n\equiv 0,1\pmod{k-1}$ and let $m=n+s\equiv0,1 (\mod k-1)$.
As $K_{n+s}=K_n\cup K_s\cup K_{n,s}$, for the existence
of a $(K_m,P_k)$-design embedding $\cB$ it is enough to find a $P_k$-decomposition of
$K_s\cup K_{n,s}$. Since $n, n+s\equiv 0,1 (\mod k-1)$, one of the following cases occurs
\begin{itemize}
\item $n=\lambda(k-1), s=\mu(k-1)\ \Rightarrow\  K_s\bigcup K_{n,s}=K_s\bigcup \lambda\mu K_{k-1,k-1}$
\item $n=\lambda(k-1), s=1+\mu(k-1)\ \Rightarrow\  K_s\bigcup K_{n,s}=K_s\bigcup K_{\lambda(k-1),k+(\mu-1)(k-1)}=$ \\
$= K_s\bigcup\lambda K_{k-1,k}\bigcup\lambda(\mu-1)K_{k-1,k-1}$
\item $n=1+\lambda(k-1), s=\mu(k-1)\ \Rightarrow \  K_s\bigcup K_{n,s}=K_s\bigcup K_{k+(\lambda-1)(k-1),\mu(k-1)}= $   \\
 $=K_s\bigcup\mu K_{k,k-1}\bigcup \mu(\lambda-1)K_{k-1,k-1}$
\item $n=1+\lambda(k-1), s=k-2+\mu(k-1) \ \Rightarrow \  K_s\bigcup K_{n,s}=K_s\bigcup  K_{1+\lambda(k-1),s}=$\\
$=K_s\bigcup K_{1,s}\bigcup K_{\lambda(k-1),s} = K_{s+1}\bigcup K_{\lambda(k-1),k-2+\mu(k-1)}=$\\
$= K_{s+1}\bigcup \lambda K_{k-1,k-2}\bigcup \lambda\mu K_{k-1,k-1}$
\end{itemize}
%\begin{center}\begin{tabular}{l|l}
%$n\equiv 0\pmod{k-1}$ & $n\equiv 0\pmod{k-1}$ \cr
%$s\equiv 0\pmod{k-1}$ & $s\equiv 1\pmod{k-1}$ \cr
%\hline
%$n\equiv 1\pmod{k-1}$ & $n\equiv 1\pmod{k-1}$ \cr
%$s\equiv 0\pmod{k-1}$ & $s\equiv k-2\pmod{k-1}$ \cr
%\end{tabular}
%\end{center}
So, to find a $P_k$-decomposition of $K_s\cup K_{n,s}$ it is sufficient to know $P_k$-decompositions of
\begin{itemize}
\item $K_s$ and $K_{s+1}$, which exist by \cite{T},
\item $K_{k-1,k-1}$, whose existence is proved in \cite{P},
\item$K_{k-1,k}$ and $K_{k-1,k-2}$, whose existence follows from Proposition \ref{bip}.
\end{itemize}
\end{proof}
\noindent
The following corollary is a straightforward consequence of Theorem \ref{thm:emb}.
\begin{cor}\label{c34}
If $n\in \cL_i \G (v)$, then
\[\{m\geq n+2\ | \ m\equiv 0,1(\mod 3)\}\subseteq \cL_i \G(v).\]
\end{cor}

\begin{rem}\label{min}
Set $\eta_i=\inf \cL_i \G(v)$. 
By Corollary \ref{c34}, $\cL_i \G(v)$ contains all admissible values $m\geq \eta_i$
apart from (possibly) $\eta_i+1$. Thus to exactly determine the spectra
it is enough to compute $\eta_i$ and ascertain if $\eta_i+1\in \cL_i \G(v)$.
\end{rem}

\section{Cycle systems and path-designs}\label{direct}
Here we shall provide some partial results on the existence
of down-links from cycle systems and path-designs
to $P_4$-designs.\\
We recall that a $k$-cycle system of order $v$, that is
a $(K_v,C_k)$-design, exists if, and only if,
$k\leq v$, $v$ is odd and $v(v-1)\equiv 0\pmod{2k}$; see
 \cite{AGA}, \cite{SJ}.

\begin{thm}
For any admissible $v$ and any $k\geq9$
\[\left\{ n\geq v- \left\lfloor\frac{k-9}{4}\right\rfloor\,\big|\, n\equiv0,1\
 (\mod{3})\right\}\subseteq \cL_2C_{k}(v)\subseteq
\cL_1C_{k}(v).\]
\end{thm}
\begin{proof}
Let $k\geq9$ and let $\cB$ be a $(K_v,C_k)$-design. Write
$t=\left\lfloor\frac{k-9}{4}\right\rfloor$.
Take $t+2$ distinct vertices $x_1,x_2,\ldots,x_t,y_1,y_2\in V(K_v)$.
Observe that it is possible to extract from each block $C\in \cB $
a $P_4$ whose vertices are different from $x_1,x_2,\ldots,x_t,y_1,y_2$,
as we are forbidding at most $4(t+1)+2=4t+6=k-3$ edges from any $k$-cycle.
Use these $P_4$'s for the down-link. Let $S$ be the image of the down-link,
considered as a subgraph of $K_{v-t}=K_v\setminus\{x_1,\ldots, x_t\}$
and remove the edges of $S$ from $K_{v-t}$ to obtain a new graph $R$.
It remains to show that $R$ admits a
$P_4$-decomposition. Observe that $|V(R)|=v-t$ and $y_1,y_2$ are two vertices of $R$ of degree
$v-t-1$. To apply Theorem \ref{every}
we have to distinguish some cases according to the congruence class
modulo $3$ of $v-t$.\\
If $v-t\equiv0(\mod\ 3)$, then $|E(R)|\equiv0(\mod\ 3)$ so the existence of a $(R,P_4)$-design is
guaranteed by Theorem \ref{every}.
Furthermore, if we add a vertex to $K_{v-t}$ we can apply Theorem \ref{every} also
to $R'=R \cup K_{1,v-t}$ since $|E(R')|\equiv 0 (\mod\ 3)$. Hence there exist down-links from $\cB$ to $(K_{v-t},P_4)$-designs
and to $(K_{v-t+1},P_4)$-designs.\\
If $v-t\equiv1(\mod\ 3)$, then $|E(R)|\equiv 0 (\mod\ 3)$, hence by
Theorem \ref{every}, there exists a $(R,P_4)$-design. So we determine down-links from $\cB$
to $(K_{v-t},P_4)$-designs.\\
Finally, if $v-t\equiv 2(\mod\ 3)$, it is sufficient to add either $u=1$ or $u=2$ vertices to
$K_{v-t}$ and then apply Theorem \ref{every} to $R''=(K_{v-t}\cup K_u \cup K_{v-t,u})\setminus S$
in order to down-link $\cB$ to $(K_{v-t+1},P_4)$-designs or to $(K_{v-t+2},P_4)$-designs,
respectively.
%The existence of a down-link from $\cB$ to these $P_4$-designs is trivial.
The statement follows from Remark \ref{min}.
\end{proof}

Arguing exactly as in the previous proof it is possible to prove the following result.

\begin{thm}
For any admissible $v$ and any $k\geq12$
\[\left\{ n\geq v- \left\lfloor\frac{k-12}{4}\right\rfloor\,\big|\, n\equiv0,1\
 (\mod{3})\right\}\subseteq \cL_2P_{k}(v)\subseteq
\cL_1P_{k}(v).\]
\end{thm}

\subsection{Small cases}
We shall now investigate in detail the spectrum problems for
$\G=C_4$ and $\G=P_5$. In order to obtain our results, we shall extensively use the method of
\emph{gluing of down-links}, introduced in \cite{BGP}.
We briefly recall the main idea: a down-link from a $(K_v,\Gamma)$-design to a  $(K_n,\Gamma')$-design
can be constructed as union of
down-links between partitions of the domain and the codomain.\\
To give designs suitable for the down-link,
 we will use difference families;
here we recall some preliminaries, for a survey see \cite{AB1}.
Let $\G$ be a graph. A set $\cal F$ of graphs isomorphic to $\G$
with vertices in $\ZZ_v$
is called a $(v,\Gamma,1)$-\emph{difference family} (DF, for short) if
the list $\Delta{\cal F}$ of differences from $\cal F$, namely the list of all
possible differences $x-y$, where $(x,y)$ is an ordered pair of adjacent vertices
of an element of $\cal F$, covers $\ZZ_v\setminus\{0\}$ exactly once.
In \cite{BPBica} it is proved that if ${\cal F}=\{B_1,\ldots,B_t\}$
is a $(v, \Gamma, 1)$-DF, then the collection of graphs $\cB=\{B_i+g\ | \ B_i\in {\cal F},\ g \in \ZZ_v\}$
is a cyclic $(K_v,\G)$-design.

\begin{lem}\label{lem:cp4}
For any $v\equiv 1,9\pmod{24}$, $v>1$, there exists a down-link from a $(K_v,C_4)$-design to a $(K_v,P_4)$-design.
For any $v\equiv 9,17\pmod{24}$ there exists a down-link from a $(K_v,C_4)$-design to a $(K_{v+1},P_4)$-design.
\end{lem}
\begin{proof}
Take $v=s+24t\geq 9$, with $s=1,9,17$,  and $V(K_v)=\ZZ_v$. Consider the set of $4$-cycles
\[{\cal C}=\left\{C^a=\left(0,a,\frac{v+1}{2},\frac{v-1}{8}+a\right)\ | \ a=1,2,\dots,\frac{v-1}{8}\right\}.\]
It is straightforward to check that
\[\Delta C^a= \pm\left\{a,\frac{v+1}{2}-a,\frac{3v+5}{8}-a,\frac{v-1}{8}+a\right\}. \]
Hence $\Delta {\cal C}=\ZZ_v\setminus\{0\}$, so, by \cite{BPBica}, the $C^a$ are the $\frac{v-1}{8}$
base blocks of a cyclic $(K_v,C_4)$-design.
The development of each base block gives $v$ different 4-cycles, from each of which we extract the edge obtained by developing
$[0,a]$. The obtained $P_4$'s will be used to define a down-link in a natural way.
 The removed edges can be connected to complete the $P_4$-decomposition of $K_v$ as follows: for each triple
 $\{[0,a+1],[0,a+2],[0,a+3]\}$, for $a\equiv 1\pmod 3$ where $a\in \{1,2,\dots,\frac{v-1}{8}\}$, consider the three
 developments   and  connect the edges $\{[i+1,a+1+(i+1)],[i,a+2+i],[i,a+3+i]\}$ obtaining the paths
 $(i+1,a+i+2,i,a+i+3)$, with $i\in \ZZ_v$.\\
   If $v\equiv 1\pmod{24}$, we have the required $P_4$-decomposition.\\
   If $v\equiv 9\pmod{24}$, we have the required $P_4$-decomposition except for the development of $[0,1]$.
 The $v$ edges of such a development can be easily  connected to give the $v$-cycle $C=(0,1,\ldots,v-1)$,
 which obviously admits a
 $P_4$-decomposition.  So, for $v\equiv 1,9\pmod{24}$, there exists a down-link from a $(K_v,C_4)$-design to a
 $(K_v,P_4)$-design.
 Under the assumption  $v\equiv 9\pmod {24}$,  $n=v+1$ is also admissible. In this case,
 add the vertex $\alpha$ to $V(K_v)$ to obtain a $K_{v+1}$ supporting the codomain of the down-link.
 Actually,  the star $S_{[\alpha;V]}$ of center $\alpha$ and external vertices
 the elements of $V(K_v)$
 has been added. Proceed as before till to the last but one step,
namely do not decompose the $v$-cycle $C$ obtained by developing $[0,1]$.
So it remains to determine a $P_4$-decomposition of the wheel $W=C\cup S_{[\alpha;V]}$.
 It is easy to see that $W$ can be decomposed into $3+8t$ copies of the graph $W'$ in Figure \ref{Ruota},
 which evidently admits a $P_4$-decomposition.

 \begin{figure}[h]
 \begin{center}
\includegraphics[width=.15\textwidth]{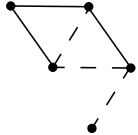}
\caption{The graph $W'$ as union of two $P_4$'s.}\label{Ruota}
\end{center}
\end{figure}
\noindent
 If $v\equiv 17\pmod {24}$, proceeding as before, we determine the required $P_4$-decomposition except for the two
 developments, say $d_1$ and $d_2$, of the edges $[0,1]$ and $[0,\frac{v-1}{8}]$. Keeping in mind that we must also  add
 a vertex, say $\alpha$,
 to the codomain, we have to arrange the edges of $d_1$, $d_2$ and $S_{[\alpha;V]}$.
 It easy to see that we can  obtain the   $P_4$'s as $[\alpha,1+i,i,\frac{v-1}{8}+i]$, for $i\in \ZZ_v$.\\
 So, for $v\equiv 9,17\pmod {24}$ there exists a down-link from a $(K_v,C_4)$-design to a $(K_{v+1},P_4)$-design.
\end{proof}

\begin{thm}\label{thm:cycle}
For any admissible $v>1$,
\begin{gather}
\cL_1C_{4}(v) =\{n\geq v\ |\ n\equiv0,1\ (\mod{3})\}; \label{L1C4} \\
\{n\geq v+2\ |\ n\equiv0,1\ (\mod{3})\}\subseteq \cL_2C_{4}(v)\subseteq\{n\geq v\ |\ n\equiv0,1\ (\mod{3})\}. \label{L2C4}
\end{gather}
\end{thm}
\begin{proof}
Let $\cB$ and $\cB'$ be, respectively, a $(K_v,C_4)$-design and a $(K_n,P_4)$-design.
Suppose that $\cB$ can be down-linked to $\cB'$.
Clearly, $n\geq v$.
Hence $\cL_2C_{4}(v)\subseteq \cL_1C_{4}(v) \subseteq  \{n\geq v\ |\ n\equiv0,1\ (\mod{3})\}$.\\
To prove the reverse inclusion in (\ref{L1C4}) observe that
a $(K_v,C_4)$-design exists if, and only if,
$v\equiv 1 \pmod 8$ and a $(K_n,P_4)$-design exists if, and only if, $n \equiv 0,1 \pmod 3$.
So it makes sense to look for a down-link from a $(K_v,C_4)$-design to a
$(K_v,P_4)$-design only for $v \equiv 1,9 \pmod {24}$. Likewise,
a down-link from a $(K_v,C_4)$-design to a
$(K_{v+1},P_4)$-design can exist only if $v \equiv 9,17 \pmod {24}$.
The existence of such down-links is proved in Lemma \ref{lem:cp4}.
The statement of (\ref{L1C4}) follows from Remark \ref{min}.
The other inclusion in (\ref{L2C4}) immediately follows from Corollary \ref{n>v+1}.
%Let $\cB$ be a $(K_v,C_4)$-design. Applying Theorem \ref{n>v+1}, we are sure that
%$\{n>v+1\ | \ n\equiv0,1\ (\mod{3})\}\subseteq \cL_2C_{4}(v)$.
%Obviously the smallest element of $\cL_2C_{4}(v)$  must be not smaller than $v$, so
%$\cL_2C_{4}(v)\subseteq \{n\geq v\ |\ n\equiv0,1\ (\mod{3})\}$.
%It remains to answer the following question: are there down-links when  $n=v$ or $n=v+1$ are admissible?\\
%Case $(a)$: $n=v$ is admissible.\\
%First of all we observe that $n=v$ is admissible if and only if $n\equiv 1,9 \pmod{24}$. If $n\equiv 9 \pmod{24}$
% $n=v+1$ is admissible too, while when $n\equiv 1\pmod{24}$, $n=v+1$ is not admissible. Anyway, in general we are not
% able to say whether or not we can give down-links to some $(K_n,P_4)$-design
%starting from any $(K_v,C_4)$-design or $(K_{v+1},C_4)$-design, but in some special cases we can, as proved in
%Lemma \ref{lem:cp4}.
%Just from Lemma \ref{lem:cp4} and Theorem \ref{thm:emb}, we know that
%$\cL_1C_{4}(v)=\{n\geq v\ |\ n\equiv0,1\ (\mod{3})\}$.\\
%Case $(b)$: $n=v$ is not admissible and  $n=v+1$ is admissible.\\
%As before, we begin by observing that  $n=v$ is not admissible if and only if $n\equiv 17 \pmod{24}$. So,
%from Lemma \ref{lem:cp4} and Theorem \ref{thm:emb}, we have again
%$\cL_1C_{4}(v)=\{n\geq v\ |\ n\equiv0,1\ (\mod{3})\}$.
\end{proof}

\begin{thm}
For any admissible $v>1$,
\begin{gather}
\cL_1P_5(v)   =   \{n\geq v-1\ |\ n\equiv0,1(\mod 3)\}; \label{L1P5}\\
\{n\geq v+2\ |\ n\equiv0,1(\mod 3)\} \subseteq \cL_2P_5(v)   \subseteq  \{n\geq v\ |\ n\equiv0,1(\mod 3)\}.\label{L2P5}
\end{gather}
\end{thm}
\begin{proof}

The first inclusion in (\ref{L2P5}) follows from Corollary \ref{n>v+1}.
In order to prove the second, it is sufficient to show that for any admissible $v$
there exists a $(K_v,P_5)$-design $\cB$ wherein no vertices can be deleted. In particular,
this is the case if each vertex of $K_v$ has degree $2$ in at least one block of $\cB$.
First of all note that in a $(K_v,P_5)$-design  there is at most one vertex with degree $1$ in each block
where it appears. Suppose that there actually exists a $(K_v,P_5)$-design $\overline{\cB}$
with a vertex $x$ as above. It is easy to see that in $\overline{\cB}$ there is  at least one block
$P^1=[x,a,b,c,d]$ such that the vertices $a,b$ and $c$ have degree two in at least another block. Let
$P^2=[x,d,e,f,g]$. By reassembling the edges of $P^1 \cup P^2$, it is possible to replace in
$\overline{\cB}$ these two paths with $P^3=[d,x,a,b,c],\ P^4=[c,d,e,f,g]$ if $c\neq f,g$
or  $P^5=[a,x,d,c,g], P^6=[a,b,c,e,d]$ if $c=f$ or $P^7=[c,d,x,a,b], P^8=[b,c,f,e,d]$ if $c=g$.
Thus we have again a $(K_v,P_5)$-design. By the assumption on $a,b,c$ all the vertices of this new design
have degree two in at least one block.\\
Now we consider Relation (\ref{L1P5}). Let $\cB$ and $\cB'$ be respectively a $(K_v,P_5)$-design
and a $(K_n,P_4)$-design.
Suppose there exists a down-link $f: \cB \rightarrow \cB'$. Clearly, $n>v-2$.
Hence, $\cL_1P_5(v) \subseteq\{n\geq v-1\ |\ n\equiv0,1(\mod 3)\}$.\\
To show the reverse inclusion in (\ref{L1P5}) we prove the actual existence of designs providing
down-links.  Since a $(K_v,P_5)$-design exists if, and only if, $v\equiv0,1(\mod 8)$
and a $(K_n,P_4)$-design exists if, and only if, $n\equiv 0,1 (\mod 3)$, it makes sense
to look for a down-link from a $(K_v,P_5)$-design to a $(K_{v-1},P_4)$-design
only if $v\equiv 1,8,16,17(\mod\ 24)$. For the same reason, it makes sense to construct
a down-link from a $(K_v,P_5)$-design to a $(K_{v},P_4)$-design only for $v\equiv 0,1,9,16(\mod 24)$.
In view of Remark \ref{min}, in order to complete the proof, we have also to provide a down-link
from a $(K_v,P_5)$-design to a $(K_{v+1},P_4)$-design for every $v\equiv 0,9(\mod 24)$.\\
To determine the necessary down-links, we analyze
a few basic cases and then apply the \emph{gluing method}. To this end, we will use
the following obvious relations in an appropriate way:
$K_{a+b}=K_a\bigcup K_b\bigcup K_{a,b}$ and $K_{a+b,c}=K_{a,c}\cup K_{b,c}$. In particular,
%$K_{\ell+24t}=K_\ell\bigcup K_{24t}\bigcup K_{\ell,24t}$; $K_{24t}=tK_{24}\bigcup
%\binom{t}{2}K_{24,24}=tK_{24}\bigcup 48\binom{t}{2}K_{3,4}$;\\
%$K_{\ell=r+s,24t}=K_{r,24t}\bigcup K_{s,24t}=tK_{r,24}\bigcup tK_{s,24}=8tK_{r,3}\bigcup 6tK_{s,4}$;\\
%$K_{\ell=rs,24t}=rK_{s,24t}=rtK_{s,24}= 6rtK_{s,4}= 8rtK_{s,3}$~~~~and so on.
\begin{eqnarray}
\nonumber K_{\ell+24t}&=&K_\ell\cup K_{24t}\cup K_{\ell,24t};\\
\nonumber K_{24t}&=&tK_{24}\cup\binom{t}{2}K_{24,24}=tK_{24}\cup 48\binom{t}{2}K_{3,4};\\
%\nonumber K_{\ell=r+s,24t} &=& K_{r,24t}\cup K_{s,24t}=tK_{r,24}\cup tK_{s,24}=8tK_{r,3}\cup 6tK_{s,4};\\
\nonumber K_{\ell=rs,24t} &=& rK_{s,24t}=rtK_{s,24}= 6rtK_{s,4}= 8rtK_{s,3}.
\end{eqnarray}
Let us now examine the possible cases.

\

 \noindent$\bullet~(K_v,P_5)$ $\rightarrow (K_{v-1},P_4)$-design with $v=\ell+24t>1$, $\ell=1,8,16,17$.\\
{\small\begin{tabular}{lllll}
&&&&\cr
   $P_5$-design           &   basic components &$\rightarrow$& basic components &$P_4$-design \cr
   of order & & & & of order \cr
$1+24t$ & $(K_{25},P_5)$, $(K_{3,4},P_5)$&&$(K_{24},P_4)$, $(K_{3,4},P_4)$ & $24t$ \cr
$8+24t$ & $(K_{8},P_5)$, $(K_{24},P_5)$ &&($K_{7},P_4)$, $(K_{24},P_4)$ &$7+24t$  \cr
&$(K_{4,3},P_5)$&& $(K_{4,3},P_4)$, $(K_{3,3},P_4)$\cr
$16+24t$ & ($K_{16},P_5)$, $(K_{24},P_5)$ &&($K_{15},P_4)$, $(K_{24},P_4)$ &$15+24t$  \cr
&$(K_{4,3},P_5)$&& $(K_{4,3},P_4)$, $(K_{3,3},P_4)$\cr
$17+24t$ & ($K_{17},P_5)$, $(K_{24},P_5)$ &&($K_{16},P_4)$, $(K_{24},P_4)$ &$16+24t$  \cr
&$(K_{4,3},P_5)$&& $(K_{4,3},P_4)$, $(K_{3,3},P_4)$\cr
\end{tabular}}

\

\noindent
$\bullet~(K_v,P_5)$ $\rightarrow (K_v,P_4)$-design with $v=\ell+24t>1$, $\ell=0,1,9,16$.\\
{\small\begin{tabular}{lllll}
&&&&\cr
   $P_5$-design             &   basic components &$\rightarrow$& basic components &$P_4$-design\cr
of order &&&& of order \cr
$24t$ & $(K_{24},P_5)$, $(K_{3,4},P_5)$&&$(K_{24},P_4)$, $(K_{3,4},P_4)$ & $24t$ \cr
$1+24t$ & $(K_{9},P_5)$, $(K_{16},P_5)$ &&$(K_{9},P_4)$, $(K_{16},P_4)$ &$1+24t$  \cr
&$(K_{24},P_5)$,$(K_{3,4},P_5)$&& $(K_{24},P_4)$,$(K_{3,4},P_4)$ \cr
$9+24t$ & $(K_{9},P_5)$, $(K_{24},P_5)$ &&($K_{9},P_4)$, $(K_{24},P_4)$ &$9+24t$  \cr
&$(K_{3,4},P_5)$&& $(K_{3,4},P_4)$\cr
$16+24t$ & $(K_{16},P_5)$, $(K_{24},P_5)$ &&$(K_{16},P_4)$, $(K_{24},P_4)$ &$16+24t$  \cr
&$(K_{3,4},P_5)$&& $(K_{3,4},P_4)$\cr
\end{tabular}}

\

\noindent
$\bullet~(K_v,P_5)$ $\rightarrow (K_{v+1},P_4)$-design with $v=\ell+24t>1$, $\ell=0,9$.\\
{\small\begin{tabular}{lllll}
&&&&\cr

$P_5$-design             &   basic components &$\rightarrow$& basic components &$P_4$-design\cr
of order &&&& of order \cr
$24t$ & $(K_{24},P_5)$, $(K_{3,4},P_5)$&&$(K_{25},P_4)$, $(K_{3,4},P_4)$ & $1+24t$ \cr
$9+24t$ & $(K_{9},P_5)$, $(K_{24},P_5)$ &&($K_{10},P_4)$, $(K_{24},P_4)$ &$10+24t$  \cr
&$(K_{3,4},P_5)$,$(K_{9,24},P_5)$&& $(K_{3,4},P_4)$,$(K_{10,24},P_4)$\cr
\end{tabular}}

\

\noindent
It is straightforward to show the existence of such basic down-links.
For instance we provide a down-link $\xi$ from  a $(K_{9,24},P_5)$-design to a $(K_{10,24},P_4)$-design.
Let $A=\{a,b,c,d,e,f,g,h,i\}$ and $B=\ZZ_{24}$, that is $K_{9,24}=K_{A,B}$.
The following are the $54$ paths of a $P_5$-decomposition of $K_{A,B}$:

\noindent
 {\scriptsize\begin{tabular}{lllll}
   $[\underline{6,a},12,b,1]$&$[\underline{1,c},12,d,6]$&$[\underline{6,e},18,f,1]$&$[\underline{1,g},12,h,0]$&$[\underline{12,i},0,a,18]$\cr
   $[\underline{7,a},13,b,2]$&$[\underline{2,c},13,d,7]$&$[\underline{7,e},19,f,2]$&$[\underline{2,g},13,h,1]$&$[\underline{13,i},1,a,19]$\cr
   $[\underline{8,a},14,b,3]$&$[\underline{3,c},14,d,8]$&$[\underline{8,e},20,f,3]$&$[\underline{3,g},14,h,2]$&$[\underline{14,i},2,a,20]$\cr
   $[\underline{9,a},15,b,4]$&$[\underline{4,c},15,d,9]$&$[\underline{9,e},21,f,4]$&$[\underline{4,g},15,h,3]$&$[\underline{15,i},3,a,21]$\cr
   $[\underline{10,a},16,b,5]$&$[\underline{5,c},16,d,10]$&$[\underline{10,e},22,f,5]$&$[\underline{5,g},16,h,4]$&$[\underline{16},i,4,a,22]$\cr
   $[\underline{11,a},17,b,0]$&$[\underline{0,c},17,d,11]$&$[\underline{11,e},23,f,0]$&$[\underline{0,g},17,h,5]$&$[\underline{17,i},5,a,23]$\cr
   \cr
   $[\underline{18,b},6,c,19]$&$[\underline{19,d},0,e,12]$&$[\underline{12,f},6,g,19]$&$[\underline{19,h},6,i,18]$&\cr
   $[\underline{19,b},7,c,20]$&$[\underline{20,d},1,e,13]$&$[\underline{13,f},7,g,20]$&$[\underline{20,h},7,i,19]$&\cr
   $[\underline{20,b},8,c,21]$&$[\underline{21,d},2,e,14]$&$[\underline{14,f},8,g,21]$&$[\underline{21,h},8,i,20]$&\cr
   $[\underline{21,b},9,c,22]$&$[\underline{22,d},3,e,15]$&$[\underline{15,f},9,g,22]$&$[\underline{22,h},9,i,21]$&\cr
   $[\underline{22,b},10,c,23]$&$[\underline{23,d},4,e,16]$&$[\underline{16,f},10,g,23]$&$[\underline{23,h},10,i,22]$&\cr
   $[\underline{23,b},11,c,18]$&$[\underline{18,d},5,e,17]$&$[\underline{17,f},11,g,18]$&$[\underline{18,h},11,i,23]$.&\cr
  \end{tabular}}

  \

 \noindent
 We obtain the  image of any $P_5$ via $\xi$ by removing the underlined edge.
  Now, to complete  the codomain, we have to add a further vertex to $A$,
 say $\alpha$,  together with  all the edges connecting $\alpha$ to the vertices of $B$.
  Thus, it remains to decompose the graph formed  by the removed edges
 together with the star of center $\alpha$ and external vertices in $B$.
 Such a $P_4$-decomposition is listed below:

 \

%\noindent
% {\scriptsize\begin{tabular}{lllll}
%  $[6,a,9,\infty]$ & $[7,a,10,\infty]$ & $[8,a,11,\infty]$ & $[1,c,4,\infty]$ & $[2,c,5,\infty]$\cr
%  $[3,c,0,\infty]$ & $[9,e,6,\infty]$ & $[10,e,7,\infty]$ & $[11,e,8,\infty]$ & $[4,g,1,\infty]$\cr
%  $[5,g,2,\infty]$ & $[0,g,3,\infty]$ & $[15,i,12,\infty]$&$[16,i,13,\infty]$ & $[17,i,14,\infty]$\cr
%  $[22,d,19,\infty]$ &  $[23,d,20,\infty]$ & $[21,d,18,\infty]$ & $[12,f,15,\infty]$ & $[13,f,16,\infty]$\cr
%  $[14,f,17,\infty]$ &  $[20,h,21,\infty]$&$[23,h,22,\infty]$&$[20,b,23,\infty]$ & $[h,18,b,21]$\cr
%   $[h,19,b,22]$ &&&&
%%
%%
%% $[6,a,9,\infty]$& $[7,a,10,\infty]$&$[8,a,11,\infty]$&$[9,e,6,\infty]$& $[10,e,7,\infty]$\cr
%% $[12,f,15,\infty]$&$[13,f,16,\infty]$&$[14,f,17,\infty]$&$[15,i,12,\infty]$&$[16,i,13,\infty]$\cr
%% $[17,i,14,\infty]$&$[11,e,8,\infty]$&$[1,c,4,\infty]$ & $[2,c,5,\infty]$&$[3,c,0,\infty]$\cr
%% $[4,g,1,\infty]$ & $[5,g,2,\infty]$&$[0,g,3,\infty]$& $[22,d,19,\infty]$&$[23,d,20,\infty]$\cr
%%% $[14,f,17,\infty]$&$[15,i,12,\infty]$&$[16,i,13,\infty]$&$[17,i,14,\infty]$& $[21,d,18,\infty]$\cr
%% $[20,h,21,\infty]$&$[23,h,22,\infty]$&$[20,b,23,\infty]$ & $[h,18,b,21]$& $[h,19,b,22]$\cr
%% $[21,d,18,\infty]$ &&&&
%    \end{tabular}}
%
\noindent
 {\scriptsize\begin{tabular}{llllll}
  $[6,a,9,\alpha]$ & $[7,a,10,\alpha]$ & $[8,a,11,\alpha]$ & $[1,c,4,\alpha]$ & $[2,c,5,\alpha]$ &\cr
  $[3,c,0,\alpha]$ & $[9,e,6,\alpha]$ & $[10,e,7,\alpha]$ & $[11,e,8,\alpha]$ & $[4,g,1,\alpha]$ &\cr
  $[5,g,2,\alpha]$ & $[0,g,3,\alpha]$ & $[15,i,12,\alpha]$&$[16,i,13,\alpha]$ & $[17,i,14,\alpha]$ &\cr
  $[22,d,19,\alpha]$ &  $[23,d,20,\alpha]$ & $[21,d,18,\alpha]$ & $[12,f,15,\alpha]$ & $[13,f,16,\alpha]$ &\cr
  $[14,f,17,\alpha]$ &  $[20,h,21,\alpha]$&$[23,h,22,\alpha]$&$[20,b,23,\alpha]$ & $[h,18,b,21]$ & \hspace{-0.2cm}
  $[h,19,b,22].$
   \end{tabular}}

\end{proof}

\end{document}